\documentclass[10pt]{article}
\usepackage{amsmath, amsthm}\usepackage{enumerate}
\usepackage{amssymb}\usepackage{bbold}
\usepackage{color}
\newtheorem{theorem}{Theorem}[section]
\newtheorem{definition}{Definition}[section] 
\newtheorem{example}{Example}[section]
\newtheorem{proposition}{Proposition}[section]

\newtheorem{corollary}{Corollary}[section]

\DeclareMathOperator{\con}{\text{\rm co}}
\DeclareMathOperator{\sol}{\text{\rm sol}}
\DeclareMathOperator{\convo}{\xrightarrow[]{o}}
\DeclareMathOperator{\convoo}{\xrightarrow[]{o_1}}
\DeclareMathOperator{\convco}{\xrightarrow[]{c-o}}
\DeclareMathOperator{\convuo}{\xrightarrow[]{uo}}
\DeclareMathOperator{\convcuo}{\xrightarrow[]{c-uo}}
\DeclareMathOperator{\convr}{\xrightarrow[]{ru}}
\DeclareMathOperator{\convcr}{\xrightarrow[]{c-ru}}

\large
\makeatletter
\renewcommand{\subsection}{\@startsection{subsection}{1}
{0pt}{3.25ex plus 1ex minus.2ex}{-1em}{\normalfont\normalsize\bf}}
\makeatother
\begin{document}

\title{{\bf Collective order convergence and collectively qualified sets of operators}}
\maketitle
\author{\centering{{Eduard Emelyanov$^{1}$\\ 
\small $1$ Sobolev Institute of Mathematics, Novosibirsk, Russia}

\abstract{Collective versions of order convergences and corresponding types 
of collectively qualified sets of operators in vector lattices are investigated.
It is proved that collectively order to norm bounded sets are bounded in the operator norm
and collectively order continuous sets are collectively order bounded.}


\large


\section{Introduction}

\hspace{4mm}
The study of collectively compact sets of operators between normed
spaces was initiated by Anselone and Palmer \cite{AP1968}
(for recent advances see \cite{E2025A,E2025B,G2024}).
In many cases it is necessary to deal with ``uniform'' or ``collective" properties 
of a set of operators between vector lattices (shortly, VLs) regarding 
the order convergence in their domains. Whereas the corresponding technique 
is rather simple for the norm convergence, the case of order convergence 
in VLs requires an additional attention. The present note is devoted to investigation 
of collective order convergence and its applications to collectively qualified sets of operators in VLs.

We abbreviate a normed space (normed lattice, Banach lattice) by NS (NL, BL). 
Throughout the text, vector spaces are real, operators are linear, 
symbol ${\cal L}(X,Y)$ stands for the space of operators between vector spaces $X$ to $Y$, symbol $x_\alpha\downarrow 0$ 
for a decreasing net in a VL such that $\inf\limits_\alpha x_\alpha=0$, symbol
$\sol(A)$ for the solid hull $\bigcup_{a\in A}[-|a|,|a|]$ of a subset $A$ of a VL $E$, 
and ${\cal L}(E,F)$ (${\cal L}_+(E,F)$, ${\cal L}_{ob}(E,F)$, ${\cal L}_{oc}(E,F)$) for
the space of linear (resp., positive, order bounded, order continuous) operators between VLs $E$ and $F$. 

We shall use the following modes of convergence in vector lattices.

\begin{definition}\label{order convergence and co} 
{\em
A net $(x_\alpha)$ in a vector lattice $E$
\begin{enumerate}[$a)$]
\item\ 
{\em order converges} to $x\in E$ 
(briefly, $x_\alpha\convo x$) 
if there exists a net $(g_\beta)$ in $E$ such that 
$g_\beta\downarrow 0$ and, for each $\beta$, 
there is an $\alpha_\beta$ with $|x_\alpha-x|\le g_\beta$ for
all $\alpha\ge\alpha_\beta$. 
\item\ 
{\em unbounded order converges} to $x\in E$ 
(briefly, $x_\alpha\convuo x$) 
if $|x_\alpha-x|\wedge u\convo 0$
for every $u\in E_+$. 
\item\ 
{\em relative uniform converges} to $x\in E$ 
(briefly, $x_\alpha\convr x$) 
if, for some $u\in E_+$ there exists an increasing sequence 
$(\alpha_n)$ of indices such that $|x_\alpha-x|\le\frac{1}{n}u$ for
all $\alpha\ge\alpha_n$. 
\end{enumerate}
An operator $T$ from a VL $E$ to a VL $F$ is called \text{\rm o}-(\text{\rm uo}-, \text{\rm ru}-)-{\em continuous} if
$Tx_\alpha\convo 0$ ($Tx_\alpha\convuo 0$, $Tx_\alpha\convr 0$) whenever $x_\alpha\convo 0$ (resp., $x_\alpha\convuo 0$, $x_\alpha\convr 0$).
}
\end{definition}

The following notions of  ``collective convergence'' are useful
in working with families of nets indexed by the same directed set.

\begin{definition}\label{COC and co}
{\em
Let ${\cal B}=\big\{\big(x^b_\alpha\big)_{\alpha\in A}\big\}_{b\in B}$ be a set of nets indexed by a directed set $A$ in a vector lattice $E$.
We say that ${\cal B}$ 
\begin{enumerate}[$a)$]
\item\ 
{\em collective \text{\rm o}-converges} to $0$ 
(briefly, ${\cal B}\convco 0$) 
if there exists a net $g_\beta\downarrow 0$ such that, for each $\beta$, 
there is an $\alpha_\beta$ with $|x^b_\alpha|\le g_\beta$ for
$\alpha\ge\alpha_\beta$ and $b\in B$. 
\item\ 
{\em collective \text{\rm uo}-converges} to $0$ 
(briefly, ${\cal B}\convcuo 0$) if \\
$\big\{\big(|x^b_\alpha|\wedge u\big)_{\alpha\in A}\big\}_{b\in B}\convco 0$
for every $u\in E_+$. 
\item\ 
{\em collective \text{\rm ru}-converges} to $0$ 
(briefly, ${\cal B}\convcr 0$)
if, for some $u\in E_+$, there exists an increasing sequence 
$(\alpha_n)$ of indices such that $|x^b_\alpha|\le\frac{1}{n}u$ for
$\alpha\ge\alpha_n$ and $b\in B$. 
\end{enumerate}}
\end{definition}

For the basic theory of VLs we refer to \cite{AB2003,AB2006,Kus2000,Mey1991}.

\section{Main results}

\subsection{\large Preliminaries.}
We begin with the following auxiliary properties of c-o, c-uo, and c-ru convergences.

\begin{proposition}\label{elem lemma}
Let ${\cal B}=\big\{\big(x^b_\alpha\big)\big\}_{b\in B}$ and
${\cal C}=\big\{\big(x^c_\alpha\big)\big\}_{c\in C}$ be nonempty 
sets of nets in a vector lattice $E$ indexed by the same directed set $A$,
and let $r\in\mathbb{R}$. The following holds.
\begin{enumerate}[$i)$]
\item
If ${\cal B}\convco 0$ and ${\cal C}\convco 0$ then 
$r\mathcal{B}=\big\{\big(rx^b_\alpha\big)\big\}_{\alpha\in A;b\in B}\convco 0$,
$$
   \sol(\mathcal{B}):=\big\{(x_\alpha): 
   (\exists b\in B)(\forall\alpha\in A)|x_\alpha|\le|x^b_\alpha|\big\}\convco 0,
$$
$\mathcal{B}\cup\mathcal{C}\convco 0$, 
$\mathcal{B}+\mathcal{C}=\big\{\big(x^b_\alpha+x^c_\alpha\big): b\in B, c\in C\big\}\convco 0$,
and $\con(\mathcal{B})\convco 0$, where
$$
   \con(\mathcal{B})=\Big\{(x_\alpha)=\Big(\sum_{k=1}^n r_kx^{b_k}_\alpha\Big):\ 
   b_k\in B,  r_k\in\mathbb{R}_+, \sum_{k=1}^n r_k=1\Big\}.
$$
\item
If ${\cal B}\convcuo 0$ and ${\cal C}\convcuo 0$ then 
$r\mathcal{B}\convcuo 0$, $\sol(\mathcal{B})\convcuo 0$, 
$\mathcal{B}\cup\mathcal{C}\convcuo 0$, 
$\mathcal{B}+\mathcal{C}\convcuo 0$, and
$\con(\mathcal{B})\convcuo 0$.
\item
If ${\cal B}\convcr 0$ and ${\cal C}\convcr 0$ then 
$r\mathcal{B}\convcr 0$, $\sol(\mathcal{B})\convcr 0$, 
$\mathcal{B}\cup\mathcal{C}\convcr 0$, 
$\mathcal{B}+\mathcal{C}\convcr 0$, and
$\con(\mathcal{B})\convcr 0$.
\end{enumerate}
Moreover, $x_\alpha\convo 0\Longleftrightarrow\{(x_\alpha)\}\convco 0$;
$x_\alpha\convuo 0\Longleftrightarrow\{(x_\alpha)\}\convcuo 0$; and $x_\alpha\convr 0\Longleftrightarrow\{(x_\alpha)\}\convcr 0$.
\end{proposition}

\begin{proof}
$i)$\
Let ${\cal B}\convco 0$ and ${\cal C}\convco 0$. Take two nets
$g_\beta\downarrow 0$ and $p_\gamma\downarrow 0$ in $E$ such that, for every 
$\beta$ and $\gamma$, there exist indices $\alpha_\beta$ and $\alpha_\gamma$ satisfying
$|x^b_\alpha|\le g_\beta$ for $\alpha\ge\alpha_\beta$ and $b\in B$,
and $|x^c_\alpha|\le p_\gamma$ for $\alpha\ge\alpha_\gamma$ and $c\in C$.

Obviously, $r\mathcal{B}\convco 0$ and $\sol(\mathcal{B})\convco 0$.

In order to show $\mathcal{B}\cup\mathcal{C}\convco 0$, consider the net
$(g_\beta+p_\gamma)_{\beta;\gamma}\downarrow 0$.
Since $\mathcal{B}\cup\mathcal{C}=\big\{\big(x^d_\alpha\big)\big\}_{\alpha\in A;d\in B\cup C}$ 
then $|x^d_\alpha|\le g_\beta+p_\gamma$ for $\alpha\ge\alpha_\beta,\alpha_\gamma$
and $d\in B\cup C$. For every $\beta$ and $\gamma$ find $\alpha_{(\beta,\gamma)}\ge\alpha_\beta,\alpha_\gamma$.
Then $|x^d_\alpha|\le g_\beta+p_\gamma$ for $\alpha\ge\alpha_{(\beta,\gamma)}$
and $d\in B\cup C$ as desired.

For proving $\mathcal{B}+\mathcal{C}\convco 0$, consider
$(g_\beta+p_\gamma)_{\beta;\gamma}\downarrow 0$,
and pick $\alpha_{(\beta,\gamma)}\ge\alpha_\beta,\alpha_\gamma$.
It follows $|x^b_\alpha+x^c_\alpha|\le g_\beta+p_\gamma$ for $\alpha\ge\alpha_{(\beta,\gamma)}$, $b\in B$, and $c\in C$ as desired. 

For proving $\con(\mathcal{B})\convco 0$, let $b_1,...,b_n\in B$ and $r_1,...,r_n\in\mathbb{R}_+$ such that $\sum\limits_{k=1}^n r_k=1$. Since, 
for $\alpha\ge\alpha_\beta$ we have
$$
   \Big|\sum_{k=1}^n r_kx^{b_k}_\alpha\Big|\le
   \sum_{k=1}^n r_k|x^{b_k}_\alpha|\le\sum_{k=1}^n r_kg_\beta=g_\beta,
$$
it follows $\con(\mathcal{B})\convco 0$.

\medskip
$ii)$\
It is a direct consequence of $i)$.

\medskip
$iii)$\
Let ${\cal B}\convcr 0$ and ${\cal C}\convcr 0$. 
Obviously, $r\mathcal{B}\convcr 0$ and $\sol(\mathcal{B})\convcr 0$.

Take $u,w\in E_+$ and two increasing sequences 
$(\alpha'_n)$ and $(\alpha''_n)$ of indices such that 
$|x^b_\alpha|\le\frac{1}{n}u$ for $\alpha\ge\alpha'_n$ and $b\in B$,
and $|x^c_\alpha|\le\frac{1}{n}w$ for $\alpha\ge\alpha''_n$ and $c\in C$.
For each $n$ find some $\alpha_n\ge\alpha'_n,\alpha''_n$.
Since $|x^d_\alpha|\le\frac{1}{n}(u+w)$ for $\alpha\ge\alpha_n$ and $d\in B\cup C$ then 
$\mathcal{B}\cup\mathcal{C}\convcr 0$.

Since $|x^b_\alpha+x^c_\alpha|\le\frac{1}{n}(u+w)$ for all $\alpha\ge\alpha_n$, 
$b\in B$, and $c\in C$, it follows $\mathcal{B}+\mathcal{C}\convcr 0$.

In order to show $\con(\mathcal{B})\convcr 0$, let 
$b_1,...,b_n\in B$ and $r_1,...,r_n\in\mathbb{R}_+$ satisfy $\sum_{k=1}^nr_k=1$. Since 
$$
   \Big|\sum_{k=1}^n r_kx^{b_k}_\alpha\Big|\le
   \sum_{k=1}^n r_k|x^{b_k}_\alpha|\le\sum_{k=1}^n \frac{r_k}{n}u=\frac{1}{n}u
$$
for all $\alpha\ge\alpha'_n$, we conclude $\con(\mathcal{B})\convcr 0$.

\medskip
The remaining part of the proof is obvious.
\end{proof}

\noindent
Definition\,\ref{COC and co} gives rise to the following notions.

\begin{definition}\label{COCS and co}
{\em
Let ${\cal T}\subseteq{\cal L}(E,F)$.\\ If $E$ and $F$ are vector lattices, the set ${\cal T}$ is
\begin{enumerate}[$a)$]
\item\ 
{\em collectively order continuous} (briefly, ${\cal T}\in\text{\bf L}_{oc}(E,F)$)
if ${\cal T}(x_\alpha)=\{(Tx_\alpha)\}_{T\in{\cal T}}\convco 0$ 
whenever $x_\alpha\convo 0$. 
\item\ 
{\em collectively \text{\rm uo}-continuous} (briefly, ${\cal T}\in\text{\bf L}_{uoc}(E,F)$)
if ${\cal T}(x_\alpha)\convcuo 0$ whenever $x_\alpha\convuo 0$.
\item\ 
{\em collectively \text{\rm ru}-continuous} (briefly, ${\cal T}\in\text{\bf L}_{rc}(E,F)$) if 
${\cal T}(x_\alpha)\convcr 0$ whenever $x_\alpha\convr 0$. 
\item\ 
{\em collectively order bounded} (briefly, ${\cal T}\in\text{\bf L}_{ob}(E,F)$) 
if the set ${\cal T}[0,b]=\{Tx: T\in{\cal T}; x\in[0,b]\}$ is order bounded for every $b\in E_+$.
\end{enumerate}
\hspace{4mm}If $E$ is a vector lattice and $F$ is a normed space, the set ${\cal T}$ is
\begin{enumerate}[]
\item[$e)$]\ 
{\em collectively order to norm bounded} (briefly, ${\cal T}\in\text{\bf L}_{onb}(E,F)$) 
if the set ${\cal T}[0,b]$ is norm bounded for every $b\in E_+$.
\end{enumerate}}
\end{definition}

Clearly, an operator $T$ is order bounded (order to norm bounded, o-, uo-, ru-continuous) iff the set
$\{T\}$ is collectively order bounded (resp., collectively order to norm bounded, o-, uo-, ru-continuous).
The following example shows that collectively o-bounded and/or collectively compact
sets of o-continuous (or, uo-continuous) functionals need not to be o-continuous (uo-continuous) collectively.

\begin{example}\label{example-1}
Consider a sequence $(f_n)$ of fuctionals on $\ell^2$
$$
   f_n\in(\ell^2)': \ \ f_n({\bf x})=x_n, 
$$
for ${\bf x}=\sum_{k=1}^\infty x_k {\bf e}_k\in\ell^2$, where ${\bf e}_k$ is
the k-th unit vector of $\ell^2$. Then $(f_n)$ is a collectively \text{\rm o}-bounded collectively compact 
sequence of \text{\rm o}- and \text{\rm uo}-continuous fuctionals
that is neither \text{\rm o}- nor \text{\rm uo}-continuous collectively.
\end{example}

\subsection{\large On continuity of order to norm bounded operators.}
It is well known that every order bounded operator from a BL to a NL is continuous (cf., \cite[Theorem 1.31]{AA2002}). 
We extend this fact in the following strong collective way.

\begin{theorem}\label{appl1}
Every collectively order to norm bounded set of operators from a Banach lattice to a normed space is bounded in the operator norm.
\end{theorem}

\begin{proof}
Let $E$ be a BL, $Y$ a NS, and ${\cal T}\in\text{\bf L}_{onb}(E,Y)$.
Assume in contrary, ${\cal T}$ is not bounded in the operator norm.
Then the set ${\cal T}(B_E)$, and hence the set ${\cal T}((B_E)_+)$ is not norm bounded. Thus, there exist
sequences $(x_n)$ in $(B_E)_+$ and $(T_n)$ in ${\cal T}$ satisfying $\|T_nx_n\|\ge n^3$ for all $n\in\mathbb{N}$.
Take $x:=\|\cdot\|-$$\sum\limits_{n=1}^\infty\frac{x_n}{n^2}\in E_+$.
Since ${\cal T}\in\text{\bf L}_{onb}(E,Y)$, there exists $M\in\mathbb{R}_+$ such that
${\cal T}([0,x])\subseteq MB_Y$.
Since $\frac{x_n}{n^2}\in[0,x]$ then $\|T_n(\frac{x_n}{n^2})\|\le M$ for every $n\in\mathbb{N}$.
We get a contradiction:
$$
   M\ge\Big\|T_n\Big(\frac{x_n}{n^2}\Big)\Big\|\ge n \ \ \ \ (\forall n\in\mathbb{N}).
$$
Therefore, the set ${\cal T}$ is norm bounded.
\end{proof}

\noindent
It is worth noting the same argument combined with the Krein -- Smulian
theorem (cf., \cite[Theorem 2.37]{AT2007}) proves norm boundedness of 
every collective order to norm bounded set of linear operators from an
ordered Banach space with a closed generating cone to a normed space.

\begin{corollary}\label{appl11}
Every linear operator from a Banach lattice to a normed space that maps order intervals to norm bounded sets is bounded.
\end{corollary}

\subsection{\large Further elementary properties of collectively qualified sets.} 

\begin{proposition}\label{appl12a}
Let $E$ and $F$ be vector lattices, and $r_1,r_2\in\mathbb{R}$. 
\begin{enumerate}[$i)$]
\item\
If ${\cal T}_1$ and ${\cal T}_2$ are nonempty subsets of $\text{\bf L}_{ob}(E,F)$, then  
$r_1{\cal T}_1+r_2{\cal T}_2=\{r_1T_1+r_2T_2: T_1\in{\cal T}_1, T_2\in{\cal T}_2\}\in\text{\bf L}_{ob}(E,F)$.
\item\
If ${\cal T}_1$ and ${\cal T}_2$ are nonempty subsets of $\text{\bf L}_{oc}(E,F)$ $($resp., $\text{\bf L}_{uoc}(E,F)$, $\text{\bf L}_{rc}(E,F)$$)$
then ${\cal T}_1\cup{\cal T}_2$ and $r_1{\cal T}_1+r_2{\cal T}_2$ are subsets of 
$\text{\bf L}_{oc}(E,F)$ $($resp., of $\text{\bf L}_{uoc}(E,F)$, $\text{\bf L}_{rc}(E,F)$$)$.
\end{enumerate}
\end{proposition}

\begin{proof}\
It follows directly from Proposition\,\ref{elem lemma}.
\end{proof}

\noindent
For each element $x$ of a VL $E$, we define the set ${\cal T}_x:=\{y^\sim: |y|\le|x|\}$ of functionals on the order dual $E^\sim$ of $E$,
where $y^\sim(f)=f(y)$ for each $f\in E^\sim={\cal L}_{ob}(E,\mathbb{R})$.

\begin{proposition}\label{appl12b}
Let $E$ be a vector lattice. Then ${\cal T}_x\in\text{\bf L}_{oc}(E^\sim,\mathbb{R})$ for each $x\in E$.
\end{proposition}

\begin{proof}\
Let $x\in E$. In order to show that ${\cal T}_x$ is collectively order continuous, 
let $f_\alpha\convo 0$ in $E^\sim$ and pick $g_\beta\downarrow 0$ in $E^\sim$ 
so that, for each $\beta$ there is an $\alpha_\beta$ with $|f_\alpha|\le g_\beta$ for
$\alpha\ge\alpha_\beta$. By the Riesz-Kantorovich formula,
$|y^\sim(f_\alpha)|=|f_\alpha(y)|\le|f_\alpha|(|y|)\le|f_\alpha|(|x|)\le g_\beta(|x|)$  
for all $|y|\le|x|$ and $\alpha\ge\alpha_\beta$. Since $g_\beta(|x|)\downarrow 0$, then
${\cal T}_x(f_\alpha)=\{(y^\sim(f_\alpha)): |y|\le|x|\}\convco 0$. 
As $f_\alpha\convo 0$ is arbitrary, we conclude ${\cal T}_x\in\text{\bf L}_{oc}(E^\sim,\mathbb{R})$.
\end{proof}

\noindent
Let $E^\delta$ be a Dedekind completion of a VL $E$, $F$ be a Dedekind complete VL, and $T\in({\cal L}_{oc})_+(E,F)$.
Define $T^\delta\in{\cal L}_{oc}(E^\delta,F)$ as the unique linear extension of an additive map $E_+\ni y\to\sup\limits_{y\ge x\in E}Tx$.
The mapping $({\cal L}_{oc})_+(E,F)\ni T\to T^\delta$
has a unique extension to a Riesz isomorphism of ${\cal L}_{oc}(E,F)$ onto ${\cal L}_{oc}(E^\delta,F)$
(see, for example, \cite[Theorem\,1.84]{AB2003}, \cite[Theorem\,3.2.3]{Kus2000}).  

\begin{proposition}\label{appl12c}
Let a VL $E$ be Archimedean, a VL $F$ Dedekind complete, and ${\cal T}\subseteq{\cal L}_+(E,F)$.
Then ${\cal T}^\delta\in\text{\bf L}_{oc}(E^\delta,F)\Longleftrightarrow{\cal T}\in\text{\bf L}_{oc}(E,F)$.
\end{proposition}

\begin{proof}\
The regularity of $E$ in $E^\delta$ (cf., \cite[Theorems 1.23 and 1.41]{AB2003}) provides the implication
${\cal T}^\delta\in\text{\bf L}_{oc}(E^\delta,F)\Longrightarrow{\cal T}\in\text{\bf L}_{oc}(E,F)$.

Now let ${\cal T}\in\text{\bf L}_{oc}(E,F)$ and $x_\alpha\convo 0$ in $E^\delta$.
Then, there is $y_\beta\downarrow 0$ in $E^\delta$ such that, for each $\beta$ there exists $\alpha_\beta$
with $|x_\alpha|\le y_\beta$ for all $\alpha\ge\alpha_\beta$. Since $E$ is majorizing in $E^\delta$ (cf., \cite[Theorem 1.41]{AB2003}),
we may assume that $(y_\beta)$ is contained in $E$. Trivially, $y_\beta\downarrow 0$ in $E$.
Since ${\cal T}\in\text{\bf L}_{oc}(E,F)$, we have $\{(Ty_\beta)\}_{T\in{\cal T}}\convco 0$, and hence
there is a net $g_\gamma\downarrow 0$ in $F$ s.t., for each $\gamma$, 
there is $\beta_\gamma$ with $|Ty_\beta|\le g_\gamma$ for all $\beta\ge\beta_\gamma$ and $T\in{\cal T}$.  
As ${\cal T}\subseteq{\cal L}_+(E,F)$, we conclude 
$|T^\delta x_\alpha|\le T^\delta|x_\alpha|\le T^\delta y_{\beta_\gamma}=Ty_{\beta_\gamma}\le g_\gamma$ 
for all $\alpha\ge\alpha_{\beta_\gamma}$ and $T\in{\cal T}$. Thus $\{(T^\delta x_\alpha)\}_{T\in{\cal T}}\convco 0$,
and hence ${\cal T}^\delta\in\text{\bf L}_{oc}(E^\delta,F)$.
\end{proof}

\subsection{\large Conditions for collectively order boundedness.}
It is long known that every \text{\rm o}-continuous operator between VLs is order bounded (see, e.g., \cite[Lemma 1.72]{AB2003}) whenever
the order convergence is understood in the sense of \cite[Definition 1.12]{AB2003} i.e., $x_\alpha\convoo x$ in a VL $E$
if there exists a net $g_\alpha\downarrow 0$ in $E$ s.t. $|x_\alpha-x|\le g_\alpha$ for all $\alpha$. 
Abramovich and Sirotkin proved in \cite{AS2005} that every \text{\rm o}-continuous operator from an Archimedean VL
to a VL is order bounded also if the order convergence is understood in the sense
of Definition \ref{order convergence and co}. The Archimedean assumption is essential in \cite[Theorem 2.1]{AS2005}
(see \cite[Example 2.1]{EEG2024}).

A short looking at the proof of \cite[Theorem 2.1]{AS2005} tell us that Abra\-mo\-vich and Sirotkin proved indeed that  
every \text{\rm ru}-continuous operator between arbitrary two VLs $E$ and $F$ is order bounded. Conversely, 
each $T\in{\cal L}_{ob}(E,F)$ is \text{\rm ru}-continuous directly by Definition \ref{order convergence and co}~$c)$
because, for every $u\in E_+$ there is $w\in F$ with $T[-u,u]\subseteq[-w,w]$, and hence 
$|x_\alpha|\le\frac{1}{n}u\Longrightarrow|Tx_\alpha|\le\frac{1}{n}w$.

The fact that $T\in{\cal L}_{ob}(E,F)$ iff $T$ is \text{\rm ru}-continuous was independently rediscovered by 
Taylor and Troitsky in \cite[Theorem 5.1]{TT2020}. It is worth mentioned that this fact can also be derived from the nonstandard criteria 
of order boundednees  \cite[Theorem 1.7.2]{E1996} (cf., \cite[Theorem 4.6.2]{E2000}).

We extend the Abramovich--Sirotkin--Taylor--Troitsky result as follows with the key idea of proof coming from \cite{AS2005}. 

\begin{theorem}\label{cocon imlies obdd}
Let $E$ and $F$ be vector lattices with $E$ Archimedean. The following statements hold.
\begin{enumerate}[$i)$]
\item
Let ${\cal T}\subseteq{\cal L}(E,F)$.
If ${\cal T}x_\alpha\convco 0$ whenever $x_\alpha\downarrow 0$, then 
${\cal T}\in\text{\bf L}_{ob}(E,F)$.
In particular, $\text{\bf L}_{oc}(E,F)\subseteq\text{\bf L}_{ob}(E,F)$.
\item
If $\dim(F)<\infty$ then $\text{\bf L}_{uoc}(E,F)\subseteq\text{\bf L}_{ob}(E,F)$.
\end{enumerate}
Moreover, $\text{\bf L}_{ob}(E,F)=\text{\bf L}_{rc}(E,F)$ for all vector lattices $E$ and $F$.
\end{theorem}

\begin{proof}
$i)$\
Let $[0,b]$ be an order interval in $E$. 
Like in the proof of \cite[Theorem 2.1]{AS2005}, let ${\cal I}=\mathbb{N}\times [0,b]$ be 
a set directed with the lexicographical order: $(m,z)\ge(n,y)$ iff either $m>n$
or $m=n$ and $z\ge y$. Let $x_{(k,y)}=\frac{1}{k}y\in [0,b]$.
Since $0\le x_{(m,y)}=\frac{1}{m}y\le\frac{1}{n}b=x_{(n,b)}$ for $(m,y)\ge(n,b)$
then $x_{(m,y)}\convr 0$, and hence $x_{(m,y)}\downarrow 0$ because $E$ is Archimedean.

By the assumption, there exists a net $g_\beta\downarrow 0$ such that, 
for every $\beta$ there exists $(m_\beta,y_\beta)$ satisfying
$|T x_{(m,y)}|\le g_\beta$ for all $(m,y)\ge(m_{\beta},y_{\beta})$ and $T\in{\cal T}$.
Pick any $g_{\beta_0}$.  Since $(m_{\beta_0}+1,y)\ge(m_{\beta_0},y_{\beta_0})$, it follows
$$
   \Big|T\Big(\frac{y}{m_{\beta_0}+1}\Big)\Big|=\big|Tx_{(m_{\beta_0}+1,y)}\big|\le g_{\beta_0}\ \ \ 
   (y\in[0,b],T\in{\cal T}).
$$ 
Then $|Ty|\le(m_{\beta_0}+1)g_{\beta_0}$ for
$y\in[0, b]$ and $T\in{\cal T}$. Since $[0,b]$
is arbitrary, the set ${\cal T}$ is collective order bounded.

\medskip
$ii)$\
Let $[0,b]\subset E$.
Take a directed set ${\cal I}$ as in the proof of $i)$,
and consider a net $x_{(k,y)}=\frac{1}{k}y$ in $[0,b]$.
Then $x_{(m,y)}\convr 0$, and hence $x_{(m,y)}\convuo 0$ as $E$ is Archimedean.
Because of $\dim(F)<\infty$, $F$ has a strong unit, say $w$. 
Since ${\cal T}\in\text{\bf L}_{uoc}(E,F)$, there exists 
a net $g_\beta\downarrow 0$ in $F$ such that, 
for every $\beta$, there exists $(m_\beta,y_\beta)$ with
$|Tx_{(m,y)}|\wedge w\le g_\beta$ for all 
$(m,y)\ge(m_{\beta},y_{\beta})$ and $T\in{\cal T}$.
As $\dim(F)<\infty$, by passing to a tail, we may assume 
$g_\beta\le\frac{w}{2}$ for all $\beta$.
So, $|Tx_{(m,y)}|\le g_\beta$ for $(m,y)\ge(m_{\beta},y_{\beta})$ and $T\in{\cal T}$.
Pick any $\beta_0$. Since $(m_{\beta_0}+1,y)\ge(m_{\beta_0},y_{\beta_0})$,
$$
   \Big|T\Big(\frac{y}{m_{\beta_0}+1}\Big)\Big|=
   \big|Tx_{(m_{\beta_0}+1,y)}\big|\le g_{\beta_0}\ \ \ 
   (y\in[0, b],T\in{\cal T}).
$$ 
Thus, $|T(y)|\le(m_{\beta_0}+1)g_{\beta_0}$ for all $y\in[0, b]$ and $T\in{\cal T}$,
and hence ${\cal T}\in\text{\bf L}_{ob}(E,F)$.

\medskip
Now, let $E$ and $F$ be arbitrary vector lattices.

\medskip
Let ${\cal T}\in\text{\bf L}_{ob}(E,F)$ and $x_\alpha\convr 0$.
Then, for some $u\in E_+$, there exists an increasing sequence 
$(\alpha_n)$ of indices with $nx_\alpha\in[-u,u]$ for all $\alpha\ge\alpha_n$.
Since ${\cal T}\in\text{\bf L}_{ob}(E,F)$ then ${\cal T}[-u,u]\subseteq[-w,w]$
for some $w\in F_+$, and hence $n{\cal T}x_\alpha\subseteq[-w,w]$ for $\alpha\ge\alpha_n$.
It follows ${\cal T}\in\text{\bf L}_{rc}(E,F)$.

Finally, let ${\cal T}\in\text{\bf L}_{rc}(E,F)$ and $[0,b]\subset E$. 
Take a directed set ${\cal I}$ and a net $x_{(k,y)}\convr 0$ as in the proof of $i)$. 
Then, for some $u\in E_+$ there exists an increasing sequence $(k_n,y_n)$ in ${\cal I}$
with $\big|T\big(\frac{y}{k}\big)\big|=|Tx_{(k,y)}|\le\frac{1}{n}u$ for
all $(k,y)\ge(k_n,y_n)$ and $T\in{\cal T}$. In particular,
$\big|T\big(\frac{y}{k_1+1}\big)\big|=\big|Tx_{(k_1+1,y)}\big|\le u$ 
for all $y\in[0, b]$ and all $T\in{\cal T}$.
This implies that $|Ty|\le(k_1+1)u$ for
for all $y\in[0, b]$ and all $T\in{\cal T}$. Since $[0,b]$
is arbitrary, ${\cal T}\in\text{\bf L}_{ob}(E,F)$.
\end{proof}

It is an open question, whether the assumption $\dim(F)<\infty$
can be dropped in Theorem\,\ref{cocon imlies obdd}\,$ii)$.
In favor of dropping this assumption, we have a relatively easy
observation that each of non order bounded operators in the textbooks 
\cite{AB2003}, \cite{AB2006}, and \cite{Za1983}
(see, \cite[Exercise 30; p.\,48]{AB2003},
\cite[Examples 2.38 and 4.73]{AB2006},
\cite[Exercise 98.7]{Za1983}) is not uo-continuous.

\medskip
From the other hand, it is well known that a positive functional need not to be uo-continuous;
e.g., $f\in(\ell^1)'$, $f({\bf x})=\sum_{k=1}^\infty x_k$,
for ${\bf x}=\sum_{k=1}^\infty x_k {\bf e}_k\in\ell^1$, where ${\bf e}_k$ is
the k-th unit vector of $\ell^1$, is not uo-continuous as $f({\bf e}_n)\equiv 1$ 
despite ${\bf e}_n\convuo 0$ in $\ell^1$.

\medskip
In the Banach lattice setting, Theorem \ref{cocon imlies obdd} has the following
application to extension of the well known fact that each o-con\-tinuous
operator from a Banach lattice to a normed lattice is continuous.

\begin{theorem}\label{appl2}
Let ${\cal T}\in\text{\bf L}_{oc}(E,F)$, where $E$ is a Banach lattice and $F$ a normed lattice. 
Then the set ${\cal T}$ is norm bounded.
\end{theorem}

\begin{proof}
By Theorem \ref{cocon imlies obdd}, ${\cal T}\in\text{\bf L}_{ob}(E,F)$,
and hence ${\cal T}\in\text{\bf L}_{onb}(E,F)$.
An application of Theorem \ref{appl1} completes the proof.
\end{proof}

\bigskip
\bigskip

{}
\end{document}